\documentclass[reqno]{amsart}
\usepackage{amssymb, amsthm, amsmath, xypic, mathrsfs}

\usepackage[pagewise]{lineno}%\linenumbers
\usepackage{color}

\calclayout

\usepackage{ifthen, xcolor}
\newboolean{commentBoolVar}
\setboolean{commentBoolVar}{true} % boolvar=true (SHOW COMMENTS) or false (HIDE COMMENTS) Need one of these in each document too
\newcommand{\colourcomment}[3]
{%
\ifthenelse{\boolean{commentBoolVar}}{{\color{#2}(#1: #3)}}{}%
}%

\theoremstyle{definition}
\newtheorem{theorem}{Theorem}[section]

\newtheorem{proposition}[theorem]{Proposition}
\newtheorem{lemma}[theorem]{Lemma}
\newtheorem{corollary}[theorem]{Corollary}

\newtheorem{remark}[theorem]{Remark}

\newcommand{\N}{\mathbb{N}}
\newcommand{\Z}{\mathbb{Z}}
\renewcommand{\bar}{\overline}

\DeclareMathOperator{\Perm}{\mbox{Perm}}
\DeclareMathOperator{\Hol}{\mbox{Hol}}
\DeclareMathOperator{\Aut}{\mbox{Aut}}

\DeclareMathOperator{\End}{\mbox{End}}

\newcommand{\id}{\mathrm{id}}

\begin{document}

\title[Bidihedral skew braces]{Bidihedral skew braces}

\author{Alan Koch}
\address{Department of Mathematics, Agnes Scott College, 141 E. College Ave., Decatur, GA 30030 USA}
\email{AKoch@agnesscott.edu}

\author{Paul J. Truman}
\address{School of Computer Science and Mathematics \\ Keele University \\ Staffordshire \\ ST5 5BG \\ UK}
\email{P.J.Truman@keele.ac.uk}



%\thanks{For the purpose of Open Access, the authors have applied a CC BY public copyright licence to any Author Accepted Manuscript (AAM) version arising from this submission.} 

\begin{abstract}
We classify skew braces with additive and multiplicative groups both isomorphic to the dihedral group $ D_{n} $. As a consequence, we obtain an alternative proof of the classification of Hopf-Galois structures of dihedral type on a dihedral Galois extension. 
\end{abstract}

\maketitle

\section{Introduction} \label{sec_introduction}

A \textit{skew (left) brace} \cite{GV17} is a triple $ (G,\cdot,\circ) $ in which $ (G,\cdot) $ and $ (G,\circ) $ are groups and the relation
\begin{equation} \label{eqn_brace_relation}
x \circ (y \cdot z) = (x \circ y) \cdot x^{-1} \cdot (x \circ z) 
\end{equation}
holds for all $ x,y,z \in G $. Here $ x^{-1} $ denotes the inverse of the element $ x $ in the \textit{additive group} $ G^{\bullet} = (G,\cdot) $; the inverse of $ x $ in the \textit{multiplicative group} $ G^{\circ} = (G,\circ) $ is denoted $ \bar{x} $. We suppress the notation $ \cdot $ wherever possible. 

Skew braces generalise \textit{braces} \cite{Ru07a}, which we may now view as skew braces with abelian additive group. Together, these objects provide an algebraic framework for constructing and studying solutions of the set-theoretic Yang-Baxter equation. They also have connections with a wide range of topics, including braids, racks, quandles,  and Hopf-Galois structures. There is therefore considerable interest in studying their structure and properties. In particular, numerous authors classify skew braces of given order or with given properties, up to the natural notion of isomorphism: a bijective function between skew braces respecting both operations. To name a few: Rump \cite{Ru07b} classifies braces with cyclic additive group, Bachiller \cite{Ba15} classifies braces of order $ p^{2} $ or $ p^{3} $ (with $ p $ prime), Acri and Bonatto \cite{AcB20} classify skew braces of order $ pq $ (with $ p,q $ distinct primes), Alabdali and Byott \cite{AB20c} classify skew braces of squarefree order, and Byott and Ferri \cite{BF25} classify braces with multiplicative group isomorphic to a dihedral group or generalised quaternion group. 

In this paper we classify skew braces $ (G,\cdot,\circ) $ such that $ G^{\bullet} \cong G^{\circ} \cong D_{n} $, the dihedral group of order $ 2n $; we call such skew braces \textit{bidihedral}. In particular, we prove

\begin{theorem} \label{thm_main_count}
Let $ n \in \N $ and let $ K_{n} = \{ k \in \Z_{n} \mid k^{2} \equiv 1 \pmod{n} \} $. Then the number of isomorphism classes of bidihedral skew braces of order $ 2n $ is 
\[ \mathrm{BD}_n = \left\{ \begin{array}{ll} 
|K_{n}| & \mbox{ if } n \equiv 1,3,5,7 \pmod{8} \\
2 |K_{n}| & \mbox{ if } n \equiv 2,4,6 \pmod{8} \\
\frac{5}{2} |K_{n}| & \mbox{ if } n \equiv 0 \pmod{8}.
\end{array} \right. \] 
\end{theorem}

Our results complement work of Kohl \cite{Ko20}, who studies the analogous question for Hopf-Galois structures; we explore the connections between his work and ours in Section \ref{sec_HGS}. 

The paper is organised as follows. In Section \ref{sec_skew_braces} we recall some fundamental definitions associated with skew braces, including $ \gamma $-functions and substructures such as \textit{ideals}. We find that a bidihedral skew brace $ (G,\cdot,\circ) $ necessarily contains a certain ideal $ (H,\cdot,\circ) $ of index $ 2 $ in which the group $ H^{\bullet} $ is cyclic and the group $ H^{\circ} $ is either cyclic or dihedral. This bifurcation forms the basis of our strategy: in each case we describe $ (H,\cdot,\circ) $ using existing classification results, and then investigate the consequences of the fact that $ (H,\cdot,\circ) $ is embedded as an ideal in a bidihedral skew brace $ (G,\cdot,\circ) $. We complete this process for $ H^{\circ} $ cyclic in Section \ref{sec_H_cyclic} and for $ H^{\circ} $ dihedral in Section \ref{sec_H_dihedral}. These sections culminate in a proof of Theorem \ref{thm_main_count}, and also provide explicit descriptions of representatives of the isomorphism classes of bidihedral skew braces. Finally, in Section \ref{sec_HGS} we use our classification to give an alternative approach to Kohl's work in Hopf-Galois theory mentioned above. 

\section{The $\gamma$-function and ideals of a skew brace} \label{sec_skew_braces}

If $ (G,\cdot,\circ) $ is a skew brace then for each $ x \in G $ the function $ \gamma_{x} : G \rightarrow G $ defined by $ \gamma_{x}(y) = x^{-1}(x \circ y) $ is an automorphism of $ G^{\bullet} $, and the function $ \gamma : G^{\circ} \rightarrow \Aut(G^{\bullet}) $ defined by $ x \mapsto \gamma_{x} $ is a homomorphism, called the \textit{$ \gamma $-function} of the skew brace. 

It follows immediately from the definition of the $ \gamma $-function that $ x \circ y = x \gamma_{x}(y) $ for all $ x,y \in G $; hence if $ (G,\cdot,\circ) $ is a skew brace then knowledge of $\gamma_x$ and either operation determines the other operation.

The $ \gamma $-function of a skew brace $ (G,\cdot,\circ) $ can also be used to characterise various substructures. A subset $ H $ of $ G $ is a called a \textit{subskew brace} to mean that it is closed with respect to both operations; this occurs if and only if it is a subgroup with respect to one of the operations with the additional property that $ \gamma_{x}(y) \in H $ for all $ x,y \in H $. A subskew brace $ H $ is called a \textit{left ideal} if it satisfies the stronger property $ \gamma_{x}(y) \in H $ for all $ x \in G $ and $ y \in H $. Finally, a left ideal $ H $ is called an  \textit{ideal} if it is normal with respect to both operations.  

Now we focus on bidihedral skew braces. We fix once and for all a presentation of the additive group $ G^{\bullet} $:
\begin{equation} \label{eqn_G_cdot_presentation}
G^{\bullet} = \langle r,s \mid r^{n} = s^{2} = e, \; srs^{-1}=r^{-1} \rangle \cong D_{n}. 
\end{equation}
Thus the elements $ r^{i} $ (with $ i=0, \ldots ,n-1 $) are the rotations, and the elements $ r^{i}s $ (with $ i=0, \ldots ,n-1 $) are the reflections, in $ G^{\bullet} $. 

We shall consider binary operations $ \circ $ on $ G $ such that $ (G,\cdot,\circ) $ is a bidihedral skew brace, and determine which of these operations yield isomorphic skew braces. If $ \circ $, $ \circ' $ are two such operations then an isomorphism of skew braces $ (G,\cdot,\circ) \rightarrow (G,\cdot,\circ') $ is a automorphism $ \theta \in \Aut(G^{\bullet}) $ such that $ \theta(x \circ y) = \theta(x) \circ' \theta(y) $ for all $ x,y \in G $;  conversely, given a skew brace $ (G,\cdot,\circ) $ and an automorphism $ \theta \in \Aut(G^{\bullet}) $ we may define $ \circ' $ by the rule $ \theta(x) \circ' \theta(y) = \theta(x \circ y) $; then $ (G,\cdot,\circ') $ is a skew brace and $ \theta : (G,\cdot,\circ) \rightarrow (G,\cdot,\circ') $ is an isomorphism of skew braces. 

We take this opportunity to recall some well-known properties of the group $ \Aut(G^{\bullet}) $, and to establish some notation. We have 
\[ \Aut(G^{\bullet}) \cong \Hol(\Z_{n}) \cong \Z_{n} \rtimes \Z_{n}^{\times} . \]
More precisely, $ \Aut(G^{\bullet}) $ is generated by the automorphism
\[ \alpha : r \mapsto r \mbox{ and } s \mapsto rs, \]
and the automorphisms
\[ \beta_{\ell} : r \mapsto r^{\ell} \mbox{ and } s \mapsto s \mbox{ with } (\ell,n)=1. \]

Earlier we noted that in a skew brace knowledge of any two of the additive operation, the multiplicative operation, or the $ \gamma $-function determines the third. Applying this observation in the context of a bidihedral skew brace $ (G,\cdot,\circ) $ (with $ G^{\bullet} $ specified as above) we obtain 

\begin{lemma} \label{lem_gamma_fn_generators}
Let $ G=(G,\cdot,\circ) $ be a bidihedral skew brace of order $ 2n $. Suppose that $ a,b \in G $ are elements that generate both $ G^{\bullet} $ and $ G^{\circ} $. Then the $ \gamma $-function of $ G $ is determined by the values
\[ \gamma_{a}(a), \; \gamma_{b}(a), \; \gamma_{a}(b), \; \gamma_{b}(b). \] 
\end{lemma}
\begin{proof}
Since $ \gamma_{x} \in \Aut(G^{\bullet}) $ for each $ x \in G $, and $ a,b $ generate $ G^{\bullet} $, the values of $ \gamma_{x} $ are determined by $ \gamma_{x}(a) $ and $ \gamma_{x}(b) $. On the other hand, since the function $ \gamma : G^{\circ} \rightarrow \Aut(G^{\bullet}) $ is a homomorphism, the values of $ \gamma $ are determined by $ \gamma_{a} $ and $ \gamma_{b} $. 
\end{proof}

We also noted above that the $ \gamma $-function of a skew brace can be used to characterise substructures such as ideals. Applying this observation in the context of a bidihedral skew brace yields the following result, which underpins the strategy employed in the remainder of the paper. 

\begin{proposition} \label{prop_H}
Suppose that $ (G,\cdot,\circ) $ is a bidihedral skew brace of order $ 2n $. 
Let $ H = \langle r \rangle_{\bullet} $, the subgroup of $ G^{\bullet} $ generated by $ r $. Then $ H $ is an ideal of the skew brace $ (G,\cdot,\circ) $. 
\end{proposition}
\begin{proof}
Since $ H^{\bullet} $ is the unique cyclic subgroup of index $ 2 $ in $ G^{\bullet} $, it is characteristic in $ G^{\bullet} $. Hence $ H^{\bullet} $ is normal in $ G^{\bullet} $ and, since $ \gamma_{x} \in \Aut(G^{\bullet}) $ for each $ x \in G $, we have $ \gamma_{x}(H)=H $ for each $ x \in G $. Hence $ H $ is also a subgroup of $ G^{\circ} $; it is normal since it has index $ 2 $. Therefore $ H $ is an ideal of $ (G,\cdot,\circ) $. 
\end{proof}

Since $ H^{\circ} $ is a subgroup of the dihedral group $ G^{\circ} $ it is either cyclic or dihedral. We consider these two cases separately over the next two sections. Each section is structured in the same way: we begin by assuming that $ (G,\cdot,\circ) $ is a bidihedral skew brace in which $ H^{\circ} $ has the given isomorphism class, use existing classification results to determine the structure of the ideal $ (H,\cdot,\circ) $, and exploit the fact that $ (H,\cdot,\circ) $ in embedded as an ideal in $ (G,\cdot,\circ) $ to derive necessary conditions on the $ \gamma $-function of $ (G,\cdot,\circ) $. We then construct explicit circle operations on $ G $ that realise each of these candidate $ \gamma $-functions, and determine the isomorphisms between the resulting skew braces. 

\section{Bidihedral skew braces with $ H^{\circ} $ cyclic} \label{sec_H_cyclic}

Let $ G^{\bullet} $ be given by the presentation \eqref{eqn_G_cdot_presentation}   and let $ H = \langle r \rangle_{\bullet} $ as in Proposition \ref{prop_H}.   We begin this section by supposing that $ (G,\cdot,\circ) $ is a bidihedral skew brace in which $ H^{\circ} $ is cyclic; later we will construct all bidihedral skew braces with this property. 

Finite skew braces in which the additive and multiplicative groups are both cyclic (\textit{bicyclic braces}) are classified by Rump \cite{Ru07b} (see also \cite{STr}). Applied  to $ H^{\bullet} $ (viewed in isolation), these results state that the binary operations $ \circ $ on $ H $ such that $ (H,\cdot,\circ) $ is a bicyclic brace are precisely those of the form 
\begin{equation} \label{eqn_rump_bicyclic}
r^{i} \circ r^{j} = r^{i+j+dij}, 
\end{equation}
where $ d $ is a divisor of $ n $ such that $ p \mid d $ for all prime numbers $ p \mid n $ and $ 4 \mid d $ if $ 4 \mid n $. 

Now we investigate the consequences of the fact that the bicyclic brace $ (H,\cdot,\circ) $ is embedded as an ideal in a bidihedral skew brace $ (G,\cdot,\circ) $.

\begin{lemma}
The element $ r $ has order $ n $ in $ G^{\circ} $, and $ s $ is a reflection in $ G^{\circ} $.
\end{lemma}
\begin{proof}
Arguing by induction we find that
\[ r^{\circ i} = r^{((1+d)^{i} - 1)/d} \mbox{ for all } i \in \N,  \]
and writing $ c=1+d $ we have
\[ \frac{ (1+d)^{i} - 1 }{d} = \frac{ c^{i} - 1 }{c - 1}. \]
We claim that $ (c^{i} - 1 )/(c - 1) \equiv 0 \pmod{n} $ if and only if $ i \equiv 0 \pmod{n} $. Let $ p^{e} $ be a maximal prime power dividing $ n $, and note that the conditions on $ d $ imply that $ c \equiv 1 \pmod{p} $ and that $ c \equiv 1 \pmod{4} $ if $ 4 \mid n $. Applying \cite[Lemma 3.1, part (ii)]{STr} in the case that $ p $ is odd and \cite[Proposition 7.1]{By07} in the case that $ p=2 $ we find that the maximal power of $ p $ dividing $ (c^{i} - 1)/(c-1) $ is equal to the maximal power of $ p $ dividing  $ v_{p}(i) $. Hence $(c^{i} - 1)/(c - 1) \equiv 0 \pmod{p^{e}} $ if and only if $ i \equiv 0 \pmod{p^{e}} $. Gathering all these prime powers together establishes the claim, and so $ r $ has order $ n $ in $ G^{\circ} $. 

Finally, since $ H^{\circ} $ is the subgroup of rotations in $ G^{\circ} $, and $ s \not \in H $, the element $ s $ is a reflection in $ G^{\circ} $. 
\end{proof}

\begin{proposition} \label{prop_bicylic_necessary} 
With the notation above, we have the following:
\begin{enumerate}
\item[]
\item The integer $ d $ satisfies
\[ d= \begin{cases} 
n \mbox{ or } n/2 & \mbox{if } 8 \mid n \\
n & \mbox{otherwise;} \end{cases} \]
\item The $ \gamma $-function of $ G $ satisfies
\begin{align*}
\gamma_{s}(r) &= r^{k} \mbox{ with } k^{2} \equiv 1 \pmod{n} \\
\gamma_{s}(s) &= s \\
\gamma_{r}(r) &= r^{1+d} \\
\gamma_{r}(s) &= r^{(1+d)k-1}s,
\end{align*}
and is determined by these values. 
\end{enumerate}
\end{proposition}
\begin{proof}
We prove both parts together. 

First, the formula \eqref{eqn_rump_bicyclic} implies immediately that
\[ \gamma_{r}(r) = r^{-1}(r \circ r) = r^{-1}r^{2+d} = r^{1+d}. \]

Second, since $ s $ is a reflection $ G^{\circ} $ we have 
\[ e = s \circ s = s \gamma_{s}(s), \]
so $ \gamma_{s}(s)= s^{-1} = s $.

Third, since $ H $ is an ideal of $ G $ the automorphism $ \gamma_{s} $ of $ G^{\bullet} $ restricts to an automorphism of $ H^{\bullet} $, so we have $ \gamma_{s}(r) = r^{k} $ for some $ k $ coprime to $ n $. Hence $ \gamma_{s}^{2}(r) = \gamma_s(r^k)=r^{k^{2}} $. But $ s \circ s = e $, so $ \gamma_{s}^{2} =  \gamma_{s \circ s} = \gamma_{e} = \id $, and so $ \gamma_{s}^{2}(r) = r $. That is, $ k^{2} \equiv 1 \pmod{n} $. 

Next we establish the stated properties of $ d $. Since $ r $ is a rotation in $ G^{\circ} $ and $ s $ is a reflection in $ G^{\circ} $ we have $ s \circ r = \bar{r} \circ s $, so $ \gamma_{s}\gamma_{r}(r) = \gamma_{r}^{-1}\gamma_{s}(r) $; that is, $ r^{k(1+d)} = r^{k(1+d)^{-1}} $. Hence $ (1+d)^{2} \equiv 1 \pmod{n} $, and so $ d(2+d) \equiv 0 \pmod{n} $. Now write $ n=2^{f}m $ with $ m $ odd and recall that $ p \mid d $ for all $ p \mid n $. We find that $ d \equiv 0 \pmod{m} $ and $ d \equiv 0 \pmod{2^{f}} $ or $ d \equiv 2^{f-1} \pmod{2^{f}} $. Combined with the fact that $ 4 \mid d $ if $ 4 \mid n $, we obtain the stated conditions on $ d $. 

Finally we consider $ \gamma_{r}(s) $. Beginning with the definition, we have
\[ \gamma_{r}(s) = r^{-1}(r \circ s) = r^{-1}(s \circ \bar{r}) = r^{-1}s \gamma_{s}(\bar{r}). \]

The conditions we have established on $ d $ imply that $ d^{2} \equiv 0 \pmod{n} $; it follows that $ \bar{r} = r^{-(1+d)} $, and so 
\[ \gamma_{r}(s) = r^{-1}s \gamma_{s}(r)^{-(1+d)} = r^{-1}s r^{-(1+d)k} = r^{(1+d)k-1}s, \]
as claimed. 
\end{proof}

Proposition \ref{prop_bicylic_necessary} gives necessary conditions for the values of the $ \gamma $-function of a bidihedral skew brace $ (G,\cdot,\circ) $ in which $ H^{\circ} $ is cyclic. In the remainder of this section we construct binary operations $ \circ $ on $ G $ that realise each of these candidate $ \gamma $-functions, and determine the isomorphism classes of the resulting skew braces. 

We recall from earlier the notation 
\[ K = K_{n} = \{ k \in \Z_{n} \mid k^{2} \equiv 1 \pmod{n} \}, \]
and the definitions of the automorphisms $ \alpha $ and $ \beta_{\ell} $ (with $ (\ell,n)=1 $) in $ \Aut(G^{\bullet}) $:
\[ \alpha : r \mapsto r \mbox{ and } s \mapsto rs, \]
and 
\[ \beta_{\ell} : r \mapsto r^{\ell} \mbox{ and } s \mapsto s. \] 

We begin by showing that we may realise all the candidate $ \gamma $-functions found in Proposition \ref{prop_bicylic_necessary} in the case $ d=n $. 

\begin{proposition} \label{prop_d_n_realise}
Let $ k \in K_{n} $ and define $ \circ $ on $ G^{\bullet} $ by
\[ r^{i}s^{u} \circ r^{j}s^{v} = \left( r^{k^{v}i}s^{u} \right)\left( r^{k^{u}j}s^{v} \right). \]
Then $ \mathscr{A}_{k} = (G,\cdot,\circ) $ is a skew brace that realises the conditions described in Proposition \ref{prop_bicylic_necessary} with $ d=n $. 
\end{proposition}
\begin{proof}
Since $ k^{2} \equiv 1 \pmod{n} $ the automorphism $ \beta_{k} $ has order $ 2 $ in $ \Aut(G^{\bullet}) $; using this automorphism we may rewrite the definition of $ \circ $ as
\[ r^{i}s^{u} \circ r^{j}s^{v} = \beta_{k}^{v}(r^{i}s^{u})\beta_{k}^{u}(r^{j}s^{v}). \]
We see immediately that $ e $ is an identity element with respect to $ \circ $ and that inverses are given by $ \bar{r^{i}s^{u}} = (r^{i}s^{u})^{-1} $. 

Next we show that $ \circ $ is associative on $ G $. To ease notation, let $ x=r^{i}s^{u} $, $ y=r^{j}s^{v} $, and $ z=r^{\ell}s^{w} $. We note that $ \beta_{k} $ has no effect on the power of $ s $ appearing in a group element (for example $ \beta_{k}(x) = r^{ki}s^{u} $) and also that the group operation in $ G^{\bullet} $ results in these powers being added modulo $ 2 $. We find 
\begin{align*}
x \circ (y \circ z) &= x \circ \beta_{k}^{w}(y)\beta_{k}^{v}(z) \\
&= \beta_{k}^{v+w}(x)\beta_{k}^{u}(\beta_{k}^{w}(y)\beta_{k}^{v}(z)) \tag{the power of $ s $ in $ \beta_{k}^{w}(y)\beta_{k}^{v}(z) $ is $ v+w $} \\
&= \beta_{k}^{v+w}(x)\beta_{k}^{u+w}(y)\beta_{k}^{u+v}(z) \\
&= \beta_{k}^{w}(\beta_{k}^{v}(x)\beta_{k}^{u}(y))\beta_{k}^{u+v}(z) \\
&= \beta_{k}^{v}(x)\beta_{k}^{u}(y) \circ z \tag{the power of $ s $ in $ \beta_{k}^{v}(x)\beta_{k}^{u}(y) $ is $ u+v $}\\
&=(x \circ y) \circ z;
\end{align*} 
thus $ \circ $ is associative on $ G $, and so $ G^{\circ} $ is a group. To show that it is dihedral, we note that we have $ r^{i} \circ r^{j} = r^{i+j} $ for all $ i,j $, so $ r $ has order $ n $ in $ G^{\circ} $, and also that $ s \circ s = e $ and that $ s \circ r = sr^{k} = r^{-k}s = r^{-1} \circ s $. 

Finally, to show that the brace relation is satisfied we begin with the right hand side:
\begin{align*}
(x \circ y)x^{-1}(x \circ z) &= \beta_{k}^{v}(x)\beta_{k}^{u}(y)x^{-1}\beta_{k}^{w}(x)\beta_{k}^{u}(z) \\
&= \beta_{k}^{v}(x)\beta_{k}^{u}(y)x^{-1}\beta_{k}^{w}(x)\beta_{k}^{u}(y)^{-1} \beta_{k}^{u}(y)\beta_{k}^{u}(z).
\end{align*}
Now $ x^{-1}\beta_{k}^{w}(x) = r^{(-1)^{u}(k^{w}-1)i} $, so $ \beta_{k}^{u}(y)x^{-1}\beta_{k}^{w}(x)\beta_{k}^{u}(y)^{-1} = r^{(-1)^{v}(-1)^{u}(k^{w}-1)i} $. We claim that this is equal to $ r^{k^{v}(-1)^{u}(k^{w}-1)i} = \beta_{k}^{v}(x^{-1}\beta_{k}^{w}(x)) $. This is certainly true if $ v=0 $ or $ w=0 $. If $ v=w=1 $ then we have $ (-1)^{v}(k^{w}-1) \equiv 1-k \pmod{n} $ and $ k^{v}(k^{w}-1) \equiv k^{2}-k \equiv 1 - k \pmod{n} $ since $ k^{2} \equiv 1 \pmod{n} $. Hence we have
\begin{align*}
(x \circ y)x^{-1}(x \circ z) &=\beta_{k}^{v}(x)\beta_{k}^{v}(x^{-1}\beta_{k}^{w}(x)) \beta_{k}^{u}(y)\beta_{k}^{u}(z) \\
&= \beta_{k}^{v+w}(x)\beta_{k}^{u}(yz) \\
&= x \circ (yz),
\end{align*}
and so  $ (G,\cdot,\circ) $ is a skew brace. We verify that this skew brace satisfies the conditions described in Proposition \ref{prop_bicylic_necessary}, with $ d=n $:
\begin{align*}
\gamma_{s}(r) &= s^{-1}(s \circ r) = s(sr^{k}) = r^{k}; \\
\gamma_{s}(s) &= s^{-1}(s \circ s) = s; \\
\gamma_{r}(r) &= r^{-1}(r \circ r) = r; \\
\gamma_{r}(s) &= r^{-1}(r \circ s) = r^{-1}(r^{k}s) = r^{k-1}s.
\end{align*}
\end{proof}

Next we study the isomorphism classes of the skew braces $ \mathscr{A}_{k} $ constructed in Proposition \ref{prop_d_n_realise}. 

\begin{proposition} \label{prop_d_n_isomorphism}
For $ k \in K_{n} $, every automorphism of $ G^{\bullet} $ is a skew brace automorphism of $ \mathscr{A}_{k} $. 
\end{proposition}
\begin{proof}
We show that the automorphisms $ \alpha $ and $ \beta_{\ell} $ (with $ (\ell,n)=1 $) of $ G^{\bullet} $ also respect $ \circ $. As in the proof of Proposition \ref{prop_d_n_realise} we let $ x=r^{i}s^{u} $, $ y=r^{j}s^{v} $, and describe $ \circ $ via the automorphism $ \beta_{k} $. 

First consider an automorphism of the form $ \beta_{\ell} $, with $ (\ell,n)=1 $. We find:
\begin{align*}
\beta_{\ell}( x \circ y ) &= \beta_{\ell} (\beta^{v}(x)\beta^{u}(y)) \\
&= \beta^{v}\beta_{\ell}(x)\beta^{u}\beta_{\ell}(y) \\
&= \beta_{\ell}(x) \circ \beta_{\ell}(y). 
\end{align*}
Hence each $ \beta_{\ell} $ respects $ \circ $. 

Next consider the automorphism $ \alpha $. We find
\begin{align*}
\alpha( x \circ y ) &= \alpha (\beta^{v}(x)\beta^{u}(y)) \\
&= r^{u+(-1)^{u}v} (\beta^{v}(x)\beta^{u}(y)),
\end{align*}
whereas 
\begin{align*}
\alpha( x ) \circ \alpha( y ) &= \beta_{k}^{v}\alpha(x)\beta_{k}^{u}\alpha(y) \\
&= \beta_{k}^{v}( r^{u}x )\beta_{k}^{u}( r^{v} y ) \\
&= r^{uk^{v} + (-1)^{u}vk^{u}} \beta_{k}^{v}( x )\beta_{k}^{u}( y ). 
\end{align*}
Comparing $ r^{u+(-1)^{u}v}  $ with $ r^{uk^{v} + (-1)^{u}vk^{u}} $, and recalling that each of $ u,v $ can only take the values $ 0,1 $, we see that $ \alpha $ also respects $ \circ $. 
\end{proof}

\begin{corollary}
For $ k,k' \in K_{n} $ we have $ \mathscr{A}_{k'} \cong \mathscr{A}_{k} $ if and only if $ k'=k $. 
\end{corollary}

We take this opportunity to study the \textit{opposites} of skew braces $ \mathscr{A}_{k} $ \cite{KT20}. In general, the opposite of a skew brace $ (G,\cdot,\circ) $ is the skew brace $ (G,\cdot,\circ)^{op} = (G,\cdot',\circ) $, where $ G^{\bullet'} $ is simply the opposite group to $ G^{\bullet} $. However, this formulation is not compatible with our desire in this paper to view $ G^{\bullet} $ as fixed and allow the binary operation $ \circ $ to vary. To accommodate this, we use the isomorphism $ \mu : G^{\bullet} \rightarrow G^{\bullet'} $ defined by $ \mu(x) = x^{-1} $ to transport the binary operation $ \circ $ to a binary operation $ \circ' $ such that $ (G,\cdot,\circ') $ is a skew brace isomorphic to $ (G,\cdot',\circ) $. Note that $ G^{\circ'} $ is not simply the opposite group to $ G^{\circ} $: in fact we have

\begin{equation} \label{eqn_opposite}
x \circ' y = \mu( \mu^{-1}(x) \circ \mu^{-1}(y)) = (x^{-1} \circ y^{-1})^{-1}. 
\end{equation}

\begin{proposition} \label{prop_A_k_opposite}
For $ k \in K_{n} $ we have $ \mathscr{A}_{k}^{op} \cong \mathscr{A}_{-k} $.  
\end{proposition}
\begin{proof}
Let $ \circ_{k} $ be the multiplicative operation in $ \mathscr{A}_{k} $ and 
$ \circ_{k}' $ be as described in \eqref{eqn_opposite}. Then we have
\begin{align*}
r^{i}s^{u} \circ_{k}' r^{j}s^{v} &= \left( r^{(-1)^{u}i}s^{u} \circ_{k} r^{(-1)^{v}j}s^{v} \right)^{-1} \\
&= \left( r^{(-1)^{u}k^{v}i}s^{u} r^{(-1)^{v}k^{u}j}s^{v} \right)^{-1} \\
&= \left( r^{(-1)^{u}k^{v}i}r^{(-1)^{u+v}k^{u}j}s^{u+v} \right)^{-1} \\
&= \left( r^{(-1)^{v}k^{v}i}r^{k^{u}j}s^{u+v} \right) \\
&= \left( r^{(-k)^{v}i}s^{u} r^{(-k)^{u}j}s^{v} \right) \\
&= r^{i}s^{u} \circ_{-k} r^{j}s^{v}.
\end{align*}
Hence $ \mathscr{A}_{k}^{op} \cong \mathscr{A}_{-k} $.
\end{proof}

Next we show that we may also realise all the candidate $ \gamma $-functions found in Proposition \ref{prop_bicylic_necessary} in the case $ d=n/2 $. Since this is only relevant in the case that $ 8 \mid n $, we impose this hypothesis for the remainder of the section. 

\begin{proposition} \label{prop_d_n_2_realise}
Suppose that $ 8 \mid n $. Let $ k \in K_{n} $ and let $ \circ $ be as described in Proposition \ref{prop_d_n_realise}. Let $ c=r^{n/2} $, and define $ \star $ on $ G $ by
\[ r^{i}s^{u} \star r^{j}s^{v} = \left( r^{i}s^{u} \circ r^{j}s^{v} \right) \circ c^{i(j+v)}. \]
Then $ \mathscr{B}_{k} = (G,\cdot,\star) $ is a skew brace that realises the conditions described in Proposition \ref{prop_bicylic_necessary} with $ d=n/2 $. 
\end{proposition}
\begin{proof}
We record a number of useful properties of the element $ c $. Clearly $ c \in Z(G^{\bullet}) $; in addition, since  $ r^{i} \circ r^{j} = r^{i+j} $ (see Proposition \ref{prop_d_n_realise}), we also have $ c \in Z(G^{\circ}) $.  
Since $ c $ has order $ 2 $ in $ G^{\bullet} $, we have $ c^{ki} = c^{i} $ for all $ i \in \Z $. As a consequence of this, we have $ x \circ c = xc $ for all $ x \in G $. We shall use all these properties frequently to simplify calculations. 

We see immediately that $ e $ is an identity element with respect to $ \star $ and that inverses are given by $ \bar{r^{i}} = r^{-i(1 + n/2)} $ and $ \bar{r^{i}s} = r^{i}s $. 

We show $ \star $ is associative on $ G $, that $ G^{\star} $ is dihedral, and that the brace relation is satisfied by exploiting the analogous properties for $ G^{\circ} $, together with the properties of $ c $ noted above.

As in the proof of Proposition \ref{prop_d_n_realise} we write $ x=r^{i}s^{u} $, $ y=r^{j}s^{v} $, and $ z=r^{\ell}s^{w} $. Then

\begin{align*}
x \star (y \star z) &= x \star (y \circ z \circ c^{j(\ell+w)} ) \\
&= x \circ y \circ z \circ c^{j(\ell+w)} \circ c^{i(j + \ell + v + w)} \tag{ $ \circ $ is associative on $ G $} \\
&= (x \circ y \circ c^{i(j+w)} ) \circ z \circ c^{(i+j)(\ell+w)} \tag{ $ c \in Z(G^{\circ}) $}\\
&= (x \star y) \star z.
\end{align*}

Hence $ \star $ is associative on $ G $, and so $ G^{\star} $ is a group. To show that it is dihedral, we calculate 
\[ r^{\star i} = \left\{ \begin{array}{ll} r^{i} & \mbox{if } i \equiv 0 \pmod{4} \\ r^{i} \circ c & \mbox{otherwise; } \end{array} \right. \]
since $ 8 \mid n $ this implies that $ r $ has order $ n $ in $ G^{\star} $. Our calculation of inverses above shows that $ s \star s = e $, and we have
\[ s \star r = s \circ r = sr^{k} = r^{-k}s \]
and
\[ \bar{r} \star s = r^{-1-n/2} \circ s \circ c = r^{-k}s. \]
Hence $ G^{\star} $ is indeed dihedral. 

Finally, we show that the brace relation is satisfied.

\begin{align*}
x \star (yz) &= x \circ (yz) \circ c^{i(j+\ell+v+w)} \\
&= (x \circ y)x^{-1}(x \circ z) \circ c^{i(j+\ell+v+w)} \tag{ $ (G,\cdot,\circ) $ is a skew brace} \\
&= (x \circ y \circ c^{i(j+v)})x^{-1}(x \circ z \circ c^{i(\ell+w)}) \circ c^{i(j+\ell+v+w)}\tag{ $ c \in Z(G^{\cdot}) \cap Z(G^{\circ}) $ }\\
&= (x \star y)x^{-1}(x \star z). 
\end{align*}

Hence $ (G,\cdot,\star) $ is a skew brace. We verify that it satisfies the conditions described in Proposition \ref{prop_bicylic_necessary}, with $ d=n/2 $: 
\begin{align*}
\gamma_{s}(r) &= s^{-1}(s \star r) = s(sr^{k}) = r^{k}; \\
\gamma_{s}(s) &= s^{-1}(s \star s) = s; \\
\gamma_{r}(r) &= r^{-1}(r \star r) = rc = r^{1 + d}; \\
\gamma_{r}(s) &= r^{-1}(r \star s) = r^{-1}(r^{k}s)c = r^{(1+d)k-1}.
\end{align*}
\end{proof}

Next we study the isomorphism classes of the skew braces $ \mathscr{B}_{k} $ constructed in Proposition \ref{prop_d_n_2_realise}. In contrast to Proposition \ref{prop_d_n_isomorphism}, in this case we do find isomorphisms  amongst the skew braces $ \mathscr{B}_{k} $.

\begin{proposition} \label{prop_d_n_2_isomorphism}
Suppose that $ 8 \mid n $. Let $ k \in K_{n} $, and let $ \star $ be defined as in Proposition \ref{prop_d_n_2_realise}. Then 
\begin{enumerate}
\item The automorphism $ \alpha $ induces an isomorphism $ \mathscr{B}_{k} \cong \mathscr{B}_{k+n/2} $, and $ \alpha^{2} $ is a skew brace automorphism of $ (G,\cdot,\star) $. 
\item the automorphisms $ \beta_{\ell} $ (with $ (\ell,n)=1 $) are skew brace automorphisms of $ (G,\cdot,\star) $. 
\end{enumerate}
\end{proposition}
\begin{proof}
As in the proof of Proposition \ref{prop_d_n_2_realise} we let $ x=r^{i}s^{u} $, $ y=r^{j}s^{v} $, and describe $ \star $ in terms of $ \circ $ and $ \cdot $. 
\begin{enumerate}
\item Write $ k' = k + n/2 $ and let $ \circ' $, $ \star' $ be the corresponding binary operations. Then
\begin{align*}
\alpha( x \star y ) &= \alpha( \left( x \circ y \right)c^{i(j+v)}) \\
&= \left( r^{i+u}s^{u} \circ r^{j+v}s^{v} \right) \alpha(c)^{i(j+v)} \\
&= \left( r^{k^{v}(i+u)}s^{u}r^{k^{u}(j+v)}s^{v} \right) c^{i(j+v)}, \\
\end{align*}
and
\begin{align*}
\alpha(x) \star' \alpha(y) &= (r^{i+u}s^{u} \circ' r^{j+v}s^{v})c^{(i+u)j}  \\
&= \left( r^{(k')^{v}(i+u)}s^{u}r^{(k')^{u}(j+v)}s^{v} \right)c^{(i+u)j} \\
&= \left( r^{k^{v}(i+u)}s^{u}r^{k^{u}(j+v)}s^{v} \right)c^{(i+u)j+v(i+u)+u(j+v)} \\
&= \left( r^{k^{v}(i+u)}s^{u}r^{k^{u}(j+v)}s^{v} \right)c^{i(j+v)}. 
\end{align*} 
Thus $ \alpha : (G,\cdot,\star) \rightarrow (G,\cdot,\star') $ is an isomorphism of skew braces. Repeating this argument shows that $ \alpha^{2} $ is a skew brace automorphism of $ \mathscr{B}_{k} $. 
\item Fix $ \ell $ with $ (\ell,n)=1 $ (thus in particular $ \ell $ is odd). We have
\begin{align*}
\beta_{\ell} ( x \star y ) &= \beta_{\ell}( \left( x \circ y \right)c^{i(j+v)}) \\
&= \left( r^{\ell i}s^{u} \circ r^{\ell j} s^{v} \right) c^{\ell i(j+v)} \tag{ by Proposition \ref{prop_d_n_isomorphism} } \\
&= \left( r^{\ell i}s^{u} \circ r^{\ell j} s^{v} \right) c^{\ell i(\ell j+v)} \tag{ $ \ell $ is odd and $ c $ has order $ 2 $} \\
&= \beta_{\ell}(x) \star \beta_{\ell}(y) 
\end{align*}
Hence $ \beta_{\ell} $ is a skew brace automorphism of $ \mathscr{B}_{k} $. 
\end{enumerate}
\end{proof}

\begin{corollary}
For $ k,k' \in K_{n} $ we have $ \mathscr{B}_{k'} \cong \mathscr{B}_{k} $ if and only if $ k'=k $ or $ k'=k + n/2 $.
\end{corollary}

Now we consider opposites, as in  Proposition \ref{prop_A_k_opposite}.

\begin{proposition} \label{prop_B_k_opposite}
For $ k \in K_{n} $ we have $ \mathscr{B}_{k}^{op} \cong \mathscr{B}_{-k} $.  
\end{proposition}
\begin{proof}
Let $ \star_{k} $ be the multiplicative operation in $ \mathscr{B}_{k} $ and 
$ \star_{k}' $ be as described in \eqref{eqn_opposite}. Then we have
\begin{align*}
r^{i}s^{u} \star_{k}' r^{j}s^{v} &= \left( r^{(-1)^{u}i}s^{u} \circ_{k} r^{(-1)^{v}j}s^{v} \circ_{k} c^{i(j+v)} \right)^{-1} \\
&= \left( (r^{(-1)^{u}i}s^{u} \circ_{k} r^{(-1)^{v}j}s^{v}) c^{i(j+v)} \right)^{-1} \\
&= \left( (r^{(-1)^{u}i}s^{u} \circ_{k} r^{(-1)^{v}j}s^{v}) \right)^{-1} c^{i(j+v)}\\
&= \left(r^{i}s^{u} \circ_{-k} r^{j}s^{v} \right) c^{i(j+v)} \tag{ by Proposition \ref{prop_A_k_opposite} } \\
&= r^{i}s^{u} \circ_{-k} r^{j}s^{v} \circ_{-k} c^{i(j+v)} \\
&= r^{i}s^{u} \star_{-k} r^{j}s^{v}.
\end{align*}
Hence $ \mathscr{B}_{k}^{op} \cong \mathscr{B}_{-k} $.
\end{proof}

Summarising the results of this section, we have

\begin{theorem} \label{thm_bicyclic_count}
The number of isomorphism classes of bidihedral skew braces $ (G,\cdot,\circ) $ of order $ 2n $ in which the the subgroup $ H^{\circ} $ is cyclic is
\[ \left\{ \begin{array}{ll} 
|K_{n}| & \mbox{ if } 8 \nmid n \\
\frac{3}{2} |K_{n}| & \mbox{ if } 8 \mid n.
\end{array} \right. \]
\end{theorem}
\begin{proof}
If $ G=(G,\cdot,\circ) $ is a bidihedral skew brace in which the subgroup $ H^{\circ} $ is cyclic then by Proposition \ref{prop_bicylic_necessary} the $ \gamma $-function of $ G $ satisfies 
\begin{align*}
\gamma_{s}(r) &= r^{k} \mbox{ with } k^{2} \equiv 1 \pmod{n} \\
\gamma_{s}(s) &= s \\
\gamma_{r}(r) &= r^{1+d} \\
\gamma_{r}(s) &= r^{(1+d)k-1}s,
\end{align*}
where $ k \in K_{n} $, and $ d=n $ or $ n/2 $ if $ 8 \mid n $ and $ d=n $ otherwise. In the case $ d=n $, the skew braces $ \mathscr{A}_{k} $ defined in Proposition \ref{prop_d_n_realise} realise each $ k \in K_{n} $, and Proposition \ref{prop_d_n_isomorphism} shows that these skew braces are pairwise nonisomorphic. Hence we obtain $ |K_{n}| $ isomorphism classes of skew braces. In the case $ d = n/2 $ (which only applies if $ 8 \mid n $) the skew braces $ \mathscr{B}_{k} $ defined in Proposition \ref{prop_d_n_2_realise} realise each $ k \in K_{n} $, and Proposition \ref{prop_d_n_2_isomorphism} shows that we obtain $ |K_{n}|/2 $ further isomorphism classes of skew braces. Hence the total number of isomorphism classes of skew braces is as stated. 
\end{proof}

\section{Bidihedral skew braces with $ H^{\circ} $ dihedral} \label{sec_H_dihedral}

Let $ G^{\bullet} $ be given by the presentation \eqref{eqn_G_cdot_presentation} and let $ H = \langle r \rangle_{\bullet} $ as in Proposition \ref{prop_H}. Mirroring the previous section, we begin by supposing that $ (G,\cdot,\ast) $ is a bidihedral skew brace in which $ H^{\ast} $ is dihedral; later we will construct all bidihedral skew braces with this property. 

Since $ |H|=n $, the assumption that $ H^{\ast} $ is dihedral implies that $ n $ is even. Finite skew braces with cyclic additive group and dihedral multiplicative group fall under the classification results established by Byott and Ferri \cite{BF25}. In particular, \cite[Proposition 5.2 and Table 1]{BF25} show that if $ n $ is a power of $ 2 $ then there is exactly one such skew brace. If $ n $ is not a power of $ 2 $ then \cite[Theorem 12.8 and Proposition 12.10]{BF25} show that there is exactly one such skew brace unless $ n = 4m $ with $ m $ odd, in which case there are three.  

Applied  to $ H^{\bullet} $ (viewed in isolation), these results imply that one possible operation $ \ast $ on $ H $ such that $ (H,\cdot,\ast) $ is a skew brace with $ H^{\ast} $ dihedral is

\begin{equation} \label{eqn_di_cyc_1}
r^{i} \ast r^{j} = r^{i+(-1)^{i}j}. 
\end{equation}

If $ n=|H| $ does not have the form $ 4m $ for some odd number $ m $ then the binary operation given in \eqref{eqn_di_cyc_1} is the only possibility. However, if $ |H|=4m $ then there are two further binary operations on $ H $ such that such that $ (H,\cdot,\ast) $ is a skew brace with $ H^{\ast} $ dihedral. These binary operation may be described explicitly by following the semidirect product constructions of \cite[Theorem 12.8]{BF25}. Using the Chinese Remainder Theorem to rewrite each $ i \in \Z $ in the form $ 4i_{4} + mi_{m} $, the two other binary operations $ \ast $ that make $ (H,\cdot,\ast)$ a skew brace with $ (H,\ast) $ dihedral are 

\begin{equation} \label{eqn_di_cyc_2}
r^{4i_{4}+mi_{m}} \ast r^{4j_{4}+mj_{m}} = r^{4(i_{4}+(-1)^{\frac{i_{m}(i_{m} \pm 1)}{2}}j_{4}) + m(i_{4} + (-1)^{i_{4}}j_{4})}. 
\end{equation}

Our first result in this section is that since $ (H,\cdot,\ast) $ is embedded as an ideal in a bidihedral skew brace $ (G,\cdot,\ast) $, neither of the binary operations described in \eqref{eqn_di_cyc_2} can occur. 

\begin{lemma}
Suppose that $ (G,\cdot,\ast) $ is a bidihedral skew brace of order $ 2n $ and that $ H^{\ast} $ is dihedral. Then the restriction of the circle operation to $ H $ is given by \eqref{eqn_di_cyc_1}. 
\end{lemma}
\begin{proof}
By the work of Byott and Ferri discussed above, if $ n \neq 4m $ with $ m $ odd then there is nothing to prove, so we may suppose that $ n = 4m $ with $ m $ odd. 

In this case, suppose for a contradiction that the restriction of $ \ast $ to $ H $ is given by
\[ r^{4i+mi'} \ast r^{4j+mj'} = r^{4(i+(-1)^{\frac{i'(i' - 1)}{2}}j + m(i' + (-1)^{i'}j')}. \]
It follows quickly from this formula that the element $ r^{4+m} $ has order $ 2m $ in $ H^{\ast} $. Since $ G^{\ast} $ contains elements of order $ 4m $ there exists $ g \in G  $ such that $ g \ast g = r^{4+m} $. 

Now consider the element $ g \ast g \ast r^{m} $. On one hand, this is equal to $ r^{4+m} \ast r^{m} $, which is equal to $ r^{4} $ by the formula above. On the other hand, we have
\begin{align*}
g \ast g \ast r^{m} &= g \ast g\gamma_{g}(r^{m}) \\
& = g \ast gr^{mu} \tag{ for some $ u $ coprime to $ n $ } \\
&= (g \ast g) g^{-1}(g \ast r^{mu}) \\
&= (g \ast g) \gamma_{g}(r^{mu}) \\
&= r^{4+m} r^{mu^{2}} \\
&= r^{4+m(u^{2}+1)}.
\end{align*}
Hence $ r^{4} = r^{4+m(u^{2}+1)} $, so $ m(u^{2}+1) \equiv 0 \pmod{n} $, so $ u^{2}+1 \equiv 0 \pmod{4} $, which is a contradiction. 

A similar argument applies to the other binary operation in \eqref{eqn_di_cyc_2}. 
\end{proof}

Hence we may assume the the restriction of $ \ast $ to $ H $ is given by \eqref{eqn_di_cyc_1}:
\[ r^{i} \ast r^{j} = r^{i+(-1)^{i}j}.  \]
We derive some further properties of the structure of $ G^{\ast} $. 

\begin{lemma}
The element $ r $ is a reflection in $ G^{\ast} $, and the element $ r^{2} $ has order $ n/2 $ in $ G^{\ast} $. Furthermore, there exists $ a \in G - H $ such that $ a \ast a = r^{2} $.
\end{lemma}
\begin{proof}
It follows quickly from \eqref{eqn_di_cyc_1} that
\[ r^{i} \ast r^{i} = \left\{ \begin{array}{cl} r^{2i} \mbox{ if $ i $ is even } \\ e \mbox{ if $ i $ is odd. } \end{array} \right. \]
This implies that $ r^{2} $ has order $ n/2 $ in $ G^{\ast} $, that $ r $ has order $ 2 $ in $ G^{\ast} $, and that 
\begin{equation} \label{eqn_di_cyc_rotation_reflection}
r \ast r^{2} = r^{-2} \ast r. 
\end{equation}
If $ n \neq 4 $ then these facts imply that $ r $ is a reflection in $ G^{ast} $ and $ r^{2} $ is a rotation in $ G^{\ast}$. Since $ G^{\ast} $ is dihedral of order $ 2n $, and $ r^{2} $ has order $ n/2 $ in $ G^{\ast} $, there exists an element $ a \in G^{\ast} $ such that $ a \ast a = r^{2} $. The formula above implies that no element of $ H $ has this property, so $ a \in G - H $, as claimed.

If $ n=4 $ then we cannot determine which of $ r $ and $ r^{2} $ is a rotation, and which is a reflection, in $ G^{\ast} $ solely from \eqref{eqn_di_cyc_rotation_reflection}. Suppose for a contradiction that $ r $ is a rotation in $ G^{\ast} $. Then there exists $ a \in G - H $ such that $ a \ast a = r $. The equation $ r \ast r = e $ implies that $ \gamma_{r}(r) = r^{-1} $, so we have $ \gamma_{a \ast a}(r) = \gamma_{a}^{2}(r) = r^{-1} $. Hence we have $ \gamma_{a}(r) = r^{u} $ with $ u^{2} \equiv -1 \pmod{4} $, which is a contradiction. Therefore $ r $ is also a reflection in $ G^{\ast} $ in this case. Similarly $ r^{3} $ is a reflection in $ G^{\ast} $, so $ r^{2} $ is a rotation in $ G^{\ast} $. As above, there exists $ a \in G - H $ such that $ a \ast a = r^{2} $. 
\end{proof}

\begin{proposition} \label{prop_di_cyc_necessary}
The $ \gamma $-function satisfies
\begin{align*}
\gamma_{a}(r) &= r^{k} \mbox{ with } k^{2} \equiv 1 \pmod{n} \\
\gamma_{a}(a) &= r^{-2}a \\
\gamma_{r}(r) &= r^{-1} \\
\gamma_{r}(a) &= r^{k-1}a,
\end{align*}
and is determined by these values. 
\end{proposition}
\begin{proof}
First, the formula \eqref{eqn_di_cyc_1} implies immediately that 
\[ \gamma_{r}(r) = r^{-1}(r \ast r) = r^{-1}. \]

Second, we have $ r^{2} = a \ast a = a \gamma_{a}(a) $, so $ \gamma_{a}(a) = a^{-1}r^{2} $; since $ a $ is a reflection in $ G^{\bullet} $ this is equal to $ r^{-2} a $. 

Next we study $ \gamma_{a}(r) $. Since $ H $ is an ideal of $ G $, the automorphism $ \gamma_{a} $ of $ G^{\bullet} $ restricts to an automorphism of $ H^{\bullet} $, and so $ \gamma_{a}(r) = r^{k} $ for some $ k $ coprime to $ n $. Hence $ \gamma_{a}^{2}(r) = r^{k^{2}} $. But $ a \ast a = r^{2} $, so $ \gamma_{a}^{2} = \gamma_{a \ast a} = \gamma_{r^{2}} $. The formula \eqref{eqn_di_cyc_1} implies that 
\[ \gamma_{r^{2}}(r) = r^{-2}(r^{2} \ast r) = r^{-2}r^{3} = r, \] 
so $\gamma_{r^{2}}(r) = r $. Hence $ r^{k^{2}} = r $, and so $ k^{2} \equiv 1 \pmod{n} $. 

Finally we consider $ \gamma_{r}(a) $. Since $ r $ is a reflection, and $ a $ is a rotation, in $ G^{\ast} $ we have
\[ \gamma_{r}(a) = r^{-1}(r \ast a) = r^{-1}(\bar{a} \ast r). \]
To compute $ \bar{a} $, note that $ a \ast a = r^{2} $ and $ r^{2} \ast r^{-2} = e $ we have
\[ \bar{a} = a \ast r^{-2} = a \gamma_{a}(r^{-2}) = ar^{-2k}. \]
Thus,
\begin{align*}
\gamma_{r}(a) &= r^{-1}(\bar{a} \ast r) \\
&= r^{-1} \bar{a} \gamma_{\bar{a}}(r) \\
&= r^{-1} ar^{-2k} \gamma_{a}^{-1}(r) \\
&= r^{-1} ar^{-2k} r^{k} \tag{ since $ k^{2} \equiv 1 \pmod{n} $} \\
&= r^{k-1} a
\end{align*}
Hence the $ \gamma $-function of $ (G,\cdot,\ast) $ satisfies the conditions given in the statement. 
\end{proof}

Analogously to the previous section, Proposition \ref{prop_di_cyc_necessary} gives necessary conditions for the values of the $ \gamma $-function of a bidihedral skew brace $ (G,\cdot,\ast) $ in which $ H^{\ast} $ is dihedral. In the remainder of this section we construct binary operations $ \ast $ on $ G $ that realise each of these candidate $ \gamma $-functions, and determine the isomorphism classes of the resulting skew braces. 

We recall again the notation 
\[ K = K_{n} = \{ k \in \Z_{n} \mid k^{2} \equiv 1 \pmod{n} \}, \]
and that the automorphisms $ \alpha $ and $ \beta_{\ell} $ (with $ (\ell,n)=1 $)  in $ \Aut(G^{\bullet}) $ are defined by 
\[ \alpha : r \mapsto r \mbox{ and } s \mapsto rs, \]
and 
\[ \beta_{\ell} : r \mapsto r^{\ell} \mbox{ and } s \mapsto s. \]

We begin by showing that for each $ k \in K_{n} $ there is a skew brace $ (G,\cdot,\ast) $ that realises the conditions described in Proposition \ref{prop_di_cyc_necessary} for a specific choice of $ a $. 

\begin{proposition} \label{prop_di_cyc_realise}
Let $ k \in K_{n} $ and consider the skew brace $ \mathscr{A}_{-k} = (G,\cdot,\circ) $. Define $ \ast  $ on $ G^{\bullet} $ by
\[ r^{i}s^{u} \ast r^{j}s^{v} = \left( r^{i}s^{u} \right) \circ \left( r^{(-1)^{i}j}s^{v} \right). \]
Then $ \mathscr{C}_{k} = (G,\cdot,\ast) $ is a skew brace that realises the conditions described in Proposition \ref{prop_di_cyc_necessary} with $ a=r^{-k}s $. 
\end{proposition}
\begin{proof}
By Proposition \ref{prop_d_n_isomorphism} the automorphism $ \beta_{-1} $ of $ G^{\bullet} $ is also an automorphism of $ G^{\circ} $; using this automorphism we may rewrite the definition of $ \ast $ as 
\[ r^{i}s^{u} \ast r^{j}s^{v} = \left( r^{i}s^{u} \right) \circ \beta_{-1}^{i}\left( r^{j}s^{v} \right). \]
We see immediately that $ e $ is an identity element with respect to $ \ast $ and that inverses are given by $ \bar{r^{i}} = r^{(-1)^{i+1}i} $ and $ \bar{r^{i}s} = r^{(-1)^{i}}s $. 

As in the proof of Proposition \ref{prop_d_n_2_realise}, we show $ \ast $ is associative on $ G $, that $ G^{\ast} $ is dihedral, and that the brace relation is satisfied by exploiting the analogous properties for $ G^{\circ} $. Let $ x=r^{i}s^{u}, y=r^{j}s^{v}, $ and $ x = r^{\ell}s^{w} $. Then
\begin{align*}
x \ast (y \ast z) &= x \ast (y \circ \beta_{-1}^{j}(z)) \\
&= x \circ \beta_{-1}^{i}(y \circ \beta_{-1}^{j}(z)) \\
&= x \circ \beta_{-1}^{i}(y) \circ \beta_{-1}^{i+j}(z), 
\end{align*}
and
\begin{align*}
(x \ast y) \ast z &= (x \circ \beta_{-1}^{i}(y)) \ast z \\
&= x \circ \beta_{-1}^{i}(y) \circ \beta_{-1}^{i+(-1)^{i}j}(z). 
\end{align*}
Since $ \beta_{-1} $ has order $ 2 $ we have $ \beta_{-1}^{i+j} = \beta_{-1}^{i+(-1)^{i}j} $; it follows that $ \ast $ is associative on $ G $, and so $ G^{\ast} $ is a group. To show that it is dihedral, we note first that 
\[ rs \ast rs = rs \circ r^{-1}s = r^{-k}sr^{k}s = r^{-2k}, \]
and that (since $ (k,n)=1 $) $ r^{-2k} $ has order $ n/2 $ in $ G^{\ast} $. Therefore $ rs $ has order $ n $ in $ G^{\ast} $. It is immediate from the definition of $ \ast $ that $ r $ has order $ 2 $ in $ G^{\ast} $. Now we have
\[ r \ast rs = r \circ r^{-1}s = r^{-k}r^{-1}s = r^{-1-k}s, \]
and
\[ \bar{rs} \ast r = r^{-1}s \ast r = r^{-1}s \circ r^{-1} = r^{-1}sr^{k} = r^{-1-k}s. \]
Hence $ r \ast rs = \bar{rs} \ast r $, and so $ G^{\ast} $ is indeed dihedral. 

Finally, to show that the brace relation is satisfied we exploit the fact that $ \mathscr{A}_{-k} $ is a skew brace. We have
\begin{align*}
x \ast (yz) &= x \circ \beta_{-1}^{i}(yz) \\
&= x \circ \beta_{-1}^{i}(y)\beta_{-1}^{i}(z) \\
&= (x \circ \beta_{-1}^{i}(y))x^{-1}(x \circ \beta_{-1}^{i}(z)) \\
&= (x \ast y)x^{-1}(x \ast z). 
\end{align*}
Hence $ (G,\cdot,\ast) $ is a skew brace. We verify that it realises the conditions described in Proposition \ref{prop_di_cyc_necessary} with $ a=r^{-k}s $. Recall that we have set $ (G,\cdot,\circ) = \mathscr{A}_{-k} $.
\begin{align*}
\gamma_{a}(r) &= (r^{-k}s)^{-1}(r^{-k}s \ast r) = (r^{-k}s)(r^{-k}s \circ r^{-1}) = r^{-k}s r^{-k} s r^{-(-k)} = r^{k} \\
\gamma_{a}(a) &= (r^{-k}s)^{-1}(r^{-k}s \ast r^{-k}s) = (r^{-k}s)(r^{-k}s \circ r^{k}s)= (r^{-k}s)(r^{(-k)^{2}}s r^{-k^{2}}s) \\
& = r^{-k-2}s = r^{-2} a \\
\gamma_{r}(r) &= r^{-1}(r \ast r) = r^{-1}(r \circ r^{-1}) = r^{-1} \\
\gamma_{r}(a) &= r^{-1}(r \ast r^{-k}s) = r^{-1}( r^{-k} r^{k}s ) = r^{-1}s = r^{k-1}a.
\end{align*}
\end{proof}

Next we show that we may transport the binary operation $ \ast $ constructed in Proposition \ref{prop_di_cyc_realise} using the automorphism $ \alpha \in \Aut(G^{\bullet}) $ to construct skew braces realising all values of $ a $ in Proposition \ref{prop_di_cyc_necessary}, and that the resulting skew braces are mutually isomorphic. 

\begin{proposition} \label{prop_di_cyc_isomorphism_1}
Let $ k \in K_{n} $ and let $ \mathscr{C}_{k} = (G,\cdot,\ast) $ be the skew brace defined in Proposition \ref{prop_di_cyc_realise}. Define $ \ast' $ on  $ G $ by
\[ x \ast' y = \alpha(\alpha^{-1}(x) \ast \alpha^{-1}(y)). \]
Then $ (G,\cdot,\ast') $ is a skew brace that realises the realises the conditions described in Proposition \ref{prop_di_cyc_necessary} with $ a=r^{-k+1}s $.
\end{proposition}
\begin{proof}
The definition of $ \ast' $ implies immediately that $ (G,\cdot,\ast') $ is a skew brace and $ \alpha : (G,\cdot,\ast) \rightarrow (G,\cdot,\ast') $ is an isomorphism. Since $ r^{-k}s \ast r^{-k}s = r^{2} $ in $ \mathscr{C}_{k} $, and $ \alpha(r^{2})=r^{2} $, we have 
\[ r^{-k+1}s \ast' r^{-k+1}s = \alpha( r^{-k}s \ast r^{-k}s ) = \alpha(r^{2}) = r^{2}. \]
Hence $ (G,\cdot,\ast') $ realises the conditions described in Proposition \ref{prop_di_cyc_necessary} with $ a=r^{-k+1}s $. 
\end{proof}

Repeating this process we see that given $ k \in K_{n} $ we may construct a skew brace realising the conditions described in Proposition \ref{prop_di_cyc_necessary} for every $ a \in G - H $, and that these skew braces are all isomorphic. 

\begin{remark} \label{rmk_C_k_alpha_2}
We note that if $ n/2 $ is even then the automorphism of $ \alpha^{n/2} \in \Aut(G^{\bullet}) $ is a skew brace automorphism of $ \mathscr{C}_{k} $, whereas if $ n/2 $ is odd then it is not. This distinction will be important in Section \ref{sec_HGS}, where we will apply our results in the context of Hopf-Galois theory. 
\end{remark}

The next proposition shows that these are the only skew brace isomorphisms amongst these skew braces. 

\begin{proposition} \label{prop_di_cyc_isomorphism_2}
Let $ k \in K_{n} $ and let $ \mathscr{C}_{k} $ be defined as in Proposition \ref{prop_di_cyc_realise}. Then the automorphisms $ \beta_{\ell} $ (with $ (\ell,n)=1 $) are skew brace automorphisms of $ \mathscr{C}_{k} $. 
\end{proposition}
\begin{proof}
By Proposition \ref{prop_d_n_isomorphism} each $ \beta_{\ell} $ is an automorphism of $ G^{\circ} $. Now for $ x=r^{i}s^{u}, y=r^{j}s^{v} \in G $ we have
\begin{align*}
\beta_{\ell}(x \ast y) &= \beta_{\ell}( x \circ \beta_{-1}^{i}(y) ) \\
&= \beta_{\ell}(x) \circ \beta_{\ell}\beta_{-1}^{i}(y) \tag{ $ \beta_{\ell} $ and $ \beta_{-1} $ commute inside $ \Aut(G^{\bullet}) $}\\
&= \beta_{\ell}(x) \circ \beta_{-1}^{\ell i}\beta_{\ell}(y) \tag{ since $ \beta_{-1} $ has order $ 2 $ and $ \ell $ is odd} \\
&= \beta_{\ell}(x) \ast \beta_{\ell}(y).
\end{align*}
Hence $ \beta_{\ell} $ is a skew brace automorphism of $ \mathscr{C}_{k} $.
\end{proof}

\begin{corollary}
For $ k,k' \in K_{n} $ we have $ \mathscr{C}_{k'} \cong \mathscr{C}_{k} $ if and only if $ k = k' $. 
\end{corollary}

Continuing our study of opposite skew braces from Proposition \ref{prop_A_k_opposite} and Proposition \ref{prop_B_k_opposite}, we have

\begin{proposition} \label{prop_B_k_opposite}
For $ k \in K_{n} $ we have $ \mathscr{C}_{k}^{op} \cong \mathscr{C}_{-k} $.  
\end{proposition}
\begin{proof}
Let $ \ast_{k} $ be the multiplicative operation in $ \mathscr{C}_{k} $ and 
$ \ast_{k}' $ be as described in \eqref{eqn_opposite}. Then we have
\begin{align*}
r^{i}s^{u} \ast_{k}' r^{j}s^{v} &= \left( r^{(-1)^{u}i}s^{u} \circ_{k}  r^{(-1)^{i+v}j}s^{v} \right)^{-1} \\
&= r^{i}s^{u} \circ_{-k}  r^{(-1)^{i}j}s^{v} \tag{by Proposition \ref{prop_A_k_opposite} } \\
&= r^{i}s^{u} \ast_{-k} r^{j}s^{v}
\end{align*}
Hence $ \mathscr{C}_{k}^{op} \cong \mathscr{C}_{-k} $.
\end{proof}

\begin{theorem} \label{thm_di_cyc_count}
Let $ n $ be even. Then there are precisely $ |K_{n}| $ isomorphism classes of bidihedral skew braces $ (G,\cdot,\circ) $ of order $ 2n $ in which the subgroup $ H^{\circ} $ is dihedral.
\end{theorem}
\begin{proof}
If $ G=(G,\cdot,\ast) $ is a bidihedral skew brace in which the subgroup $ H^{\circ} $ is dihedral then by Proposition \ref{prop_di_cyc_necessary} the $ \gamma $-function of $ G $ satisfies 
\begin{align*}
\gamma_{a}(r) &= r^{k} \mbox{ with } k^{2} \equiv 1 \pmod{n} \\
\gamma_{a}(a) &= r^{-2}a \\
\gamma_{r}(r) &= r^{-1} \\
\gamma_{r}(a) &= r^{k-1}a,
\end{align*}
for some $ k \in K_{n} $ and $ a \in G - H $. Given $ k \in K_{n} $, Proposition \ref{prop_d_n_2_realise} shows that there is a skew brace that realises $ k $ for some $ a \in G - H $, Proposition \ref{prop_di_cyc_isomorphism_1} shows that all such $ a $ yield isomorphic skew braces, and Proposition \ref{prop_di_cyc_isomorphism_2} shows that there are no further isomorphisms amongst these skew braces. Therefore we obtain $ |K_{n}| $ isomorphism classes of skew braces with these properties. 
\end{proof}

Combining Theorem \ref{thm_di_cyc_count} and Theorem \ref{thm_bicyclic_count}, we obtain Theorem \ref{thm_main_count}:

\begin{proof}[Proof of Theorem \ref{thm_main_count}.]
For all $ n \in \N $ the skew braces $ \mathscr{A}_{k} $ form a family of $ |K_{n}| $ pairwise nonisomorphic bidihedral skew braces. 

If $ n $ is odd then these account for all bidihedral skew braces. 

If $ n $ is even then the skew braces $ \mathscr{C}_{k} $ form a further family of $ |K_{n}| $ pairwise nonisomorphic bidihedral skew braces. 

If $ n \equiv 0 \pmod{8} $ then in addition the skew braces $ \mathscr{B}_{k} $ form a further family of $ |K_{n}|/2 $ pairwise nonisomorphic bidihedral skew braces. 

\end{proof}

%\section{Solutions of the set-theoretic Yang-Baxter equation} \label{sec_YBE}

\section{Hopf-Galois structures of dihedral type on dihedral Galois extensions} \label{sec_HGS}

One of the most fruitful avenues of research in skew brace theory is their connection with \textit{Hopf-Galois theory} (see in particular \cite{HAGMT}, \cite{ST23a}). In this section we show how the results of the previous sections can be used to give a new approach to the classification of \textit{Hopf-Galois structures} on dihedral field extensions.  

A Hopf-Galois structure on a finite extension of fields $ E/F $ consists of an $ F $-Hopf algebra $ H $ and a $ F $-linear action of $ H $ on $ E $ such that $ H $-module algebra and the natural $ F $-linear map $ E \otimes_{F} H \rightarrow \End_{F}(E) $ is an isomorphism. Hopf-Galois structures can be used to generalise the techniques of classical Galois theory to extensions that are inseparable or non-normal, and also have applications in algebraic number theory (see \cite{TWE}, \cite{HAGMT}.)

In the case in which $ E/F $ is a finite separable extension, a theorem of Greither and Pareigis \cite{GP87} classifies the Hopf-Galois structures admitted by $ E/F $ in purely group theoretic terms. Specialising further to the case in which $ E/F $ is a Galois extension, this theorem states that there is a bijection between Hopf-Galois structures on $ E/F $ and certain regular subgroups $ N $ of the group $ \Perm(G) $; the isomorphism class of such a subgroup is known as the \textit{type} of the corresponding Hopf-Galois structure. Numerous authors use this framework to classify Hopf-Galois structures on Galois extensions of given order, or with given properties: see for example \cite{Ch96}, \cite{By96}, \cite{Ko98}, \cite{CC99}, \cite{By04c}, \cite{By07}, \cite{Ko13}. In particular, in \cite{Ko20} Kohl enumerates Hopf-Galois structures of type $ D_{n} $ on a $ D_{n} $-Galois extension, as follows: 

\begin{theorem}[Kohl, 2020] \label{thm_Kohl}
Let $ K_{n} = \{ k \in \Z_{n} \mid k^{2} \equiv 1 \pmod{n} \} $. Then the number of Hopf-Galois structures of type $ D_{n} $ on a $ D_{n} $-Galois extension is
\[
\begin{cases}
|K_{n}| & \text{if } n \equiv 1,3,5,7 \pmod{8} \\[6pt]
(n + 1)\,|K_{n}| & \text{if } n \equiv 2,6 \pmod{8} \\[6pt]
\left(\dfrac{n}{2} + 1\right)\,|K_{n}| & \text{if } n \equiv 4 \pmod{8} \\[6pt]
\left(\dfrac{n}{2} + 2\right)\,|K_{n}| & \text{if } n \equiv 0 \pmod{8}.
\end{cases}
\]
\end{theorem}

The Greither-Pareigis classification is the linchpin of the connection between Hopf-Galois structures on Galois field extensions and the theory of skew braces. As initially observed by Bachiller \cite{Ba16}, and developed by Byott and Vendramin in the appendix to \cite{SV18}, a Galois extension with Galois group $ G^{\circ} $ admits a Hopf Galois structure of type $ N $ if and only if there exists a binary operation $ \cdot $ on $ G $ such that $ G^{\bullet} $ is a group isomorphic to $ N $ and $ (G,\cdot,\circ) $ is a skew brace. Thus Kohl's enumeration of Hopf-Galois structures of type $ D_{n} $ on a $ D_{n} $-Galois extension is connected to our enumeration of isomorphism classes of bidihedral skew braces (Theorem \ref{thm_main_count}). However, the correspondence between Hopf-Galois structures on Galois extensions and isomorphism classes of skew braces is not bijective: by \cite[Corollary 3.2]{KT23} and \cite[Corollary 2.4]{NZ19}, a given representative $ (G,\cdot,\circ) $ of an isomorphism class of skew braces yields 
\begin{equation} \label{eqn_HGS_correction_factor}
\mathrm{HG}(G,\cdot,\circ) = \frac{| \Aut(G^{\circ}) |}{ | \Aut(G,\cdot,\circ) | } 
\end{equation}
distinct Hopf-Galois structures on a Galois extension with Galois group $ G^{\circ} $. Combining this correction factor with Theorem \ref{thm_main_count} we can give an alternative proof of Theorem \ref{thm_Kohl}

\begin{proof}[Proof of Theorem \ref{thm_Kohl} using skew braces]
First we compute the correction factor \eqref{eqn_HGS_correction_factor} corresponding to each family of skew braces $ \mathscr{A}_{k}, \mathscr{B}_{k}, $ and $ \mathscr{C}_{k} $. 
\begin{itemize}
\item By Proposition \ref{prop_d_n_isomorphism} we have $ | \Aut(\mathscr{A}_{k}) | = | \Aut(G^{\circ}) | $ for each $ k $, so $ \mathrm{HG}(\mathscr{A}_{k})=1 $ for each $ k $. 
\item By Proposition \ref{prop_d_n_2_isomorphism} we have $ | \Aut(\mathscr{B}_{k}) | = | \Aut(G^{\circ}) | / 2 $, so $ \mathrm{HG}(\mathscr{B}_{k})= 2 $ for each $ k $. 
\item By Proposition \ref{prop_di_cyc_isomorphism_1}, Proposition \ref{prop_di_cyc_isomorphism_2}, and Remark \ref{rmk_C_k_alpha_2} we have $ | \Aut(\mathscr{C}_{k}) | = \varphi(n) $ (respectively, $ 2\varphi(n) $) if $ n \equiv 2,6 \pmod{8} $ (respectively, $ n \equiv 0,4 \pmod{8} $).  Hence $ \mathrm{HG}(\mathscr{C}_{k})= n $ (respectively, $ \mathrm{HG}(\mathscr{C}_{k})= n/2 $) for each $ k $. 
\end{itemize}

Now we refer to Theorem \ref{thm_main_count} and consider the possible congruence classes of $ n $ modulo $ 8 $ in turn.

\begin{itemize}
\item If $ n \equiv 1,3,5, $ or $ 7 $ modulo $ 8 $ then the $ |K_{n}| $ isomorphism classes of bidihedral skew braces of order $ 2n $ are represented by $ \mathscr{A}_{k} $ for $ k \in K_{n} $. Since $ \mathrm{HG}(\mathscr{A}_{k})=1 $ for each $ k $, there are $ |K_{n}| $ Hopf-Galois structures in this case. 

\item If $ n \equiv 2 $ or $ 6 $ modulo $ 8 $ then the $ 2|K_{n}| $ isomorphism classes of bidihedral skew braces of order $ 2n $ are represented by $ \mathscr{A}_{k} $ (with $ \mathrm{HG}(\mathscr{A}_{k})=1$) and $ \mathscr{C}_{k} $  (with $ \mathrm{HG}(\mathscr{C}_{k})=n $) for $ k \in K_{n} $. Hence there are $ (n+1)|K_{n}| $ Hopf-Galois structures in this case. 

\item If $ n \equiv 4 \pmod{8} $ then the $ 2|K_{n}| $ isomorphism classes of bidihedral skew braces of order $ 2n $ are represented by $ \mathscr{A}_{k} $ (with $ \mathrm{HG}(\mathscr{A}_{k})=1$) and $ \mathscr{C}_{k} $ (with $ \mathrm{HG}(\mathscr{C}_{k})= n/2 $) for $ k \in K_{n} $. Hence there are $ (n/2+1)|K_{n}| $ Hopf-Galois structures in this case. 

\item Finally, if $ n \equiv 0 \pmod{8} $ then the $ (5/2)|K_{n}| $ isomorphism classes of bidihedral skew braces of order $ 2n $ are represented by $ \mathscr{A}_{k} $ (with $ \mathrm{HG}(\mathscr{A}_{k})=1$) and $ \mathscr{C}_{k} $ (with $ \mathrm{HG}(\mathscr{C}_{k})= n/2 $) for $ k \in K_{n} $  and by $ \mathscr{B}_{k} $ (with $ \mathrm{HG}(\mathscr{B}_{k})= 2 $) for $ k \in K_{n} $ and $ k < n/2 $. Hence there are $ (n/2+2)|K_{n}| $ Hopf-Galois structures in this case. 
\end{itemize}

\end{proof}

\bibliography{Omahabib}

@Article{ST23a,
  author    = {Stefanello, L. and Trappeniers, S.},
  journal   = {Bull. Lond. Math. Soc.},
  title     = {On the connection between {H}opf--{G}alois structures and skew braces},
  year      = {2023},
  number    = {4},
  pages     = {1726--1748},
  volume    = {55},
  publisher = {Wiley Online Library},
}

@book{HAGMT,
	author = {L. N. Childs and C. Greither and K. P. Keating and A. Koch and T. Kohl and P. J. Truman and R. Underwood},
	publisher = {{A}merican Mathematical Society},
	series = {Mathematical Surveys and Monographs},
	title = {{H}opf algebras and {G}alois module theory},
	volume = {260},
	year = {2021}}

@article{AcB20,
	author = {E. Acri and M. Bonatto},
	journal = {Comm. Algebra},
	number = {5},
	pages = {1872-1881},
	title = {Skew braces of size $pq$},
	volume = {48},
	year = {2020}}

@article{AB20c,
	author = {A. A. Alabdali and N. P. Byott},
	journal = {J. Algebra Appl.},
	pages = {2150128},
	title = {Skew braces of squarefree order},
	year = {2020}}

@article{Ba15,
	author = {D. Bachiller},
	journal = {J. Pure Appl. Algebra},
	number = {8},
	pages = {3568-3603},
	title = {Classification of braces of order $p^3$},
	volume = {219},
	year = {2015}}

@article{Ba16,
	author = {D. Bachiller},
	journal = {J. Algebra},
	pages = {160-176},
	title = {Counterexample to a conjecture about braces},
	volume = {453},
	year = {2016}}

@article{By04c,
	author = {N. P. Byott},
	journal = {J. Pure Appl. Algebra},
	number = {1-3,2.2},
	pages = {45-57},
	title = {{H}opf-{G}alois structures on {G}alois field extensions of degree $pq$},
	volume = {188},
	year = {2004}}

@article{By96,
	author = {N. P. Byott},
	journal = {Comm. Algebra},
	pages = {3217--3228, 3705},
	title = {Uniqueness of {H}opf {G}alois structure of separable field extensions},
	volume = {24},
	year = {1996}}

@Article{By07,
  author        = {N. P. Byott},
  journal       = {J. Algebra},
  title         = {{H}opf-{G}alois structures on almost cyclic field extensions of 2-power degree},
  year          = {2007},
  number        = {1},
  pages         = {351--371},
  volume        = {318},
}

@article{CC99,
	author = {S. Carnahan and L. N. Childs},
	journal = {J. Algebra},
	pages = {81-92},
	title = {Counting {H}opf {G}alois structures on non-abelian {G}alois extensions},
	volume = {218},
	year = {1999}}

@article{Ch96,
	author = {L. N. Childs},
	journal = {New York J. Math.},
	pages = {86-102},
	title = {{H}opf-{G}alois structures on degree $ p^{2} $ cyclic extensions of local fields},
	volume = {2},
	year = {1996}}

@book{TWE,
	author = {L. N. Childs},
	publisher = {{A}merican Mathematical Society},
	series = {Mathematical Surveys and Monographs},
	shorthand = {TWE},
	title = {Taming Wild Extensions: {H}opf algebras and local {G}alois module theory},
	volume = {80},
	year = {2000}}

@article{GP87,
	author = {C. Greither and B. Pareigis},
	journal = {J. Algebra},
	pages = {239-258},
	title = {{H}opf {G}alois theory for separable field extensions},
	volume = {106},
	year = {1987}}

@article{GV17,
	author = {L. Guarneri and L. Vendramin},
	journal = {Math. Comp.},
	number = {307},
	pages = {2519-2534},
	title = {Skew braces and the {Y}ang-{B}axter equation},
	volume = {86},
	year = {2017}}

@article{KT20,
	author = {A. Koch and P. J. Truman},
	journal = {J. Algebra},
	pages = {218-235},
	title = {Opposite skew left braces and applications},
	volume = {546},
	year = {2020}}

@article{KT23,
	author = {A. Koch and P. J. Truman},
	journal = {J. Algebra Appl.},
	number = {5},
	title = {Skew left braces and isomorphism problems for {H}opf-{G}alois structures on {G}alois extensions},
	volume = {22},
	year = {2023}}

@article{Ko13,
	author = {T. Kohl},
	journal = {Algebra and Number Theory},
	number = {9},
	pages = {2203-2240},
	title = {Regular permutation groups of order $mp$ and {H}opf {G}alois structures},
	volume = {7},
	year = {2013}}

@article{Ko20,
	author = {T. Kohl},
	journal = {J. Algebra},
	pages = {93-115},
	title = {Enumerating Dihedral {H}opf-{G}alois Structures Acting on Dihedral Extensions},
	volume = {542},
	year = {2020}}

@article{Ko98,
	author = {T. Kohl},
	journal = {J. Algebra},
	pages = {525--546},
	title = {Classification of the {H}opf {G}alois structures on prime power radical extensions},
	volume = {207},
	year = {1998}}

@article{NZ19,
	author = {K. {Nejabati Zenouz}},
	journal = {J. Algebra},
	pages = {187-225},
	title = {Skew braces and {H}opf-{G}alois structures of {H}eisenberg type},
	volume = {524},
	year = {2019}}

@article{Ru07a,
	author = {W. Rump},
	journal = {J. Algebra},
	pages = {153--170},
	title = {Braces, radical rings, and the quantum {Y}ang-{B}axter equation},
	volume = {307},
	year = {2007}}

@article{Ru07b,
	author = {W. Rump},
	journal = {J. Pure Appl. Algebra},
	number = {3},
	pages = {671--685},
	title = {Classification of cyclic braces},
	unidentified = {209 no. --685. MR2298848},
	volume = {209},
	year = {2007}}

@article{SV18,
	author = {A. Smoktunowicz and L. Vendramin},
	journal = {J. Comb. Algebra},
	number = {1},
	pages = {47-86},
	title = {On skew braces (with an appendix by {N}. {B}yott and {L}. {V}endramin)},
	volume = {2},
	year = {2018}}

@Article{BF25,
  author   = {N. P. Byott and F. Ferri},
  journal  = {J. Algebra},
  title    = {On the number of quaternion and dihedral braces and {H}opf–{G}alois structures},
  year     = {2025},
  issn     = {0021-8693},
  pages    = {72-102},
  volume   = {665},
  doi      = {https://doi.org/10.1016/j.jalgebra.2024.11.011},
  url      = {https://www.sciencedirect.com/science/article/pii/S0021869324006185},
}

@Unpublished{Str,
  author = {H. Simpson and P.J. Truman},
  note   = {arXiv:2606.23004},
  title  = {Bicyclic biskew braces},
}
\bibliographystyle{plain}

\end{document}